\documentclass[11pt]{article}
\pdfoutput=1

\usepackage{lineno,lmodern} 

\usepackage[dvipsnames]{xcolor}

\def\beq{\begin{equation} }\def\eeq{\end{equation} }\def\1{\mathbf{1}}

\usepackage[framemethod=default]{mdframed}
\usepackage{caption}

\usepackage{indentfirst}
\usepackage{bm, mathrsfs, graphics,float,amssymb,amsmath,subeqnarray,setspace,graphicx,amsthm,epstopdf,subfigure, enumerate, color}
\usepackage[utf8]{inputenc}
\usepackage[colorlinks,
linkcolor=red,
anchorcolor=blue,
citecolor=blue
]{hyperref}
\usepackage{natbib}

\usepackage{fullpage}
\numberwithin{equation}{section}

\newtheorem{lemma}{Lemma}
\newtheorem{theorem}{Theorem}
\newtheorem{proposition}{Proposition}

\newtheorem{corollary}[theorem]{Corollary}
\newtheorem{remark}{Remark}

\ifx\assumption\undefined
\newtheorem{assumption}{Assumption}
\fi

\newcommand{\RR}{\mathbb{R}}

\usepackage{bbm}

\newcommand{\tu}{\tilde{u}}

\newcommand{\sr}{\mathrm{SR1}}

\usepackage{multirow}
\usepackage{tablefootnote}
\usepackage{colortbl}
\usepackage{hhline}

\usepackage{algorithm}
\usepackage{algorithmic}

\newcommand{\norm}[1]{\left\|#1\right\|}
\newcommand{\dotprod}[1]{\left\langle #1\right\rangle}
\def\tr{\mathrm{tr}}
\newcommand{\tG}{\widetilde{G}}
\newcommand{\msr}{\texttt{SR1\_CS}}
\newcommand{\msrs}{\texttt{SR1\_CS} }

\begin{document}
\title{
Explicit Superlinear Convergence Rates of  The SR1 Algorithm
}

\author{	Haishan Ye 
\thanks{Equal Contribution.}
\thanks{School of Management; Xi'an Jiaotong University; 
	\texttt{hsye\_cs@outlook.com, xiangyuchang@xjtu.edu.cn};}
\and
Dachao Lin\footnotemark[1]
\thanks{Academy for Advanced Interdisciplinary Studies;
	Peking University;
	\texttt{lindachao@pku.edu.cn};}
\and
Zhihua Zhang 
\thanks{School of Mathematical Sciences;
	Peking University;
	\texttt{zhzhang@math.pku.edu.cn}.}
\and 
Xiangyu Chang \footnotemark[2]
}
\date{\today}
\maketitle

\begin{abstract}
We study the convergence rate of the famous Symmetric Rank-1 (SR1) algorithm which has wide applications in different scenarios. 
Although it has been extensively investigated, SR1  still lacks a non-asymptotic superlinear rate compared with other quasi-Newton methods such as DFP and BFGS.
In this paper we address this problem.
Inspired by the recent work on explicit convergence analysis of quasi-Newton methods, we obtain the first explicit non-asymptotic rates of superlinear convergence for the vanilla SR1 methods with correction strategy to achieve the numerical stability.
Specifically, the vanilla SR1 with the correction strategy achieves the rates of the form $\left(\frac{4n\ln(e\kappa) }{k}\right)^{k/2}$ for general smooth strongly-convex functions where $k$ is the iteration counter,  $\kappa$ is the condition number of the objective function and $n$ is the dimension of the problem. 
For the quadratic function, the vanilla SR1 algorithm can find the optima of the objective function at most $n$ steps. 
\end{abstract}

\section{Introduction}

In this paper, we study an important kind of classical quasi-Newton method named SR1 for the smooth unconstrained optimization.
Similar to other quasi-Newton methods (e.g., DFP and BFGS), SR1 attempts to replace the exact Hessian in the Newton method with some approximation and the update of approximation only involves the gradients of the objective function.
Due to only using the gradients, quasi-Newton commonly can achieve much lower computation complexity compared with the exact Newton method.
The detailed introduction to quasi-Newton such as SR1, DFP, and BFGS can be found in Chapter 6 of \citep{nocedal2006numerical}.
And randomized quasi-Newton methods can be found in \citep{byrd2016stochastic,moritz2016linearly,gower2016stochastic,gower2017randomized,kovalev2020fast}.

Because of the low computation cost per iteration and fast convergence rate, quasi-Newton has been extensively studied, especially its convergence rate.
Many works in the literature have shown that quasi-Newton algorithms can achieve superlinear convergence rates \citep{nocedal2006numerical,broyden1970convergence,broyden1970convergence1,fletcher1970new,shanno1970conditioning,powell1971convergence,dixon1972quasi,dixon1972quasi1,broyden1973local,goldfarb1970family,wei2004superlinear}.

However, the superlinear convergence rates achieved in these works are only asymptotic, that is, the current works simply show that the ratio of successive residuals in the method tends to zero as the number of iterations goes to infinity, without providing any specific bounds on the corresponding rate of convergence.
Recently, \citet{rodomanov2021greedy} justified the first explicit rates for greedy quasi-Newton methods which employ the basis vectors and greedily select to maximize a certain measure of progress for Hessian approximation, opposed to classical quasi-Newton methods which use the difference of successive iterates for updating Hessian approximation.
\citet{lin2021faster} presented faster explicit rates for greedy SR1 and BFGS and their randomized version.
Inspired by \citet{rodomanov2021greedy}, the explicit superlinear convergence rates for restricted Broyden family quasi-Newton methods were first given in \citet{rodomanov2021rates}.
Specifically, they showed that BFGS can achieve the superlinear convergence rate of the form $\left(\frac{n\kappa}{k}\right)^{k/2}$, where $k$ is the iteration counter, $n$ is the dimension of the problem, $\kappa$ is the condition number of the objective function. 
Later, \citet{rodomanov2021new} provided improved convergence rates of restricted Broyden family quasi-Newton methods.
At the same time, \citet{jin2020non} gave the explicit superlinear convergence rates of DFP and BFGS based on the Frobenius-norm potential function, which was different from potential functions used in \citep{rodomanov2021rates,rodomanov2021new}.
These works fully exploit many existing tools developed for analyzing convergence rates of quasi-Newton methods such as different kinds of potential functions \citep{byrd1987global,byrd1989tool,byrd1992behavior}.

Though there have been some works in the literature that gave the explicit superlinear convergence rates of quasi-Newton methods,
the superlinear rate of SR1 is still mysterious.
Current explicit superlinear convergence rates only hold for  quasi-Newton methods in the restricted Broyden family where algorithms can be represented by the convex combination of DFP and BFGS \citep{rodomanov2021rates,rodomanov2021new,jin2020non}.
Unfortunately, the SR1 method does not belong to the restricted Broyden family.
In fact, the convergence properties of the SR1 method are not as well understood as those of
the BFGS method.
To the best of our knowledge, no local superlinear results similar to the ones of BFGS and DFP have been established, except the results for quadratic functions and $k$-steps superlinear convergence conditioned on several assumptions \citep{nocedal2006numerical}.

The hardness of analyzing of SR1 algorithm is due to the fact that there maybe exist some steps of the SR1 update being ill-defined.
Even for a convex quadratic function, there may be steps on which there is no symmetric rank-1 update that satisfies the secant equation \citep{nocedal2006numerical}.
And this will cause numerical instabilities and the breakdown of SR1.
These problems make it hard to describe the convergence dynamics of SR1.

In this paper, we focus on studying the explicit superlinear convergence rate of the classical SR1 algorithm which only involves the gradients of the objective function.
First, we propose a novel method to conquer the ill-definedness of the vanilla SR1 update.
Interesting, for the quadratic function, the restriction that initial Hessian approximation matrix $G_0$ satisfies that $ G_0 \succeq \nabla^2 f(x) $ will give a well-defined SR1 algorithm, where $\nabla^2 f(x)$ is the Hessian matrix.
For general strongly convex functions, not only requiring the restriction that $G_0 \succeq \nabla^2 f(x_0)$, where $x_0$ is the initial point, the correction strategy is also introduced.
In this paper, we refer to the SR1 algorithm with correction strategy as \msr.
Based on the numerical stable SR1 algorithm \msr, we show that \msrs can converge  superlinearly  for quadratic and general strongly convex functions and we also provide  explicit superlinear convergence rates. 

We summarize our contribution as follows.
\begin{enumerate}
	\item We propose a novel SR1 algorithm named \msrs which is numerically stable and its updates are well-defined.
	We also empirically validate the numerical stability of \msrs compared with the vanilla SR1 algorithm.
	\item We prove that \msrs achieves an explicit superlinear convergence rate  $\left(\frac{4n\ln (e\kappa) }{k}\right)^{k/2}$ for general smooth strongly-convex functions. 
	We also show that vanilla SR1 algorithm with initial Hessian approximation matrix $G_0$ satisfying $ G_0 \succeq \nabla^2 f(x) $ can achieve the superlinear convergence rate and will find the optima of the objective function at most $n$ steps for quadratic functions.
	\item Our paper provides the first explicit superlinear convergence rate for the SR1 type algorithm that only  uses the difference of successive iterates for updating Hessian approximation.
	To the best of our knowledge, no similar rate has been obtained  for SR1 algorithms before our work. 
\end{enumerate}

\subsection{Organization} 
In the remainder of this paper, we first introduce the notation used throughout this paper.
Section~\ref{sec:sr1} gives the update formula for the SR1 method and provides several important properties of the SR1 update.
Section~\ref{sec:quad} obtains the explicit superlinear convergence rates of SR1 for the quadratic function based on two different potential functions.
Section~\ref{sec:gen} extends the convergence rate of SR1 for the quadratic function to the general smooth strongly convex function.
Section~\ref{sec:exp} validates the numerical stability of \msr.
We compare convergence rates of SR1 derived in this paper with the greedy SR1 and existing quasi-newton methods in Section~\ref{sec:disc}.
Finally, we conclude our work in Section~\ref{sec:conc}.

\subsection{Notation}
In this paper, we consider the following unconstrained optimization problem
\begin{equation}  \label{eq:prob}
\min_{x\in\RR^n} f(x),
\end{equation}
where $f(x)$ is further assumed to be a convex and smooth function whose gradient and Hessian exist and are denoted as $\nabla f(x)$ and $\nabla^2 f(x)$, respectively. 
Denote by $\mu > 0$ the strong convexity parameter of $f$, and
by $L >0$ the Lipschitz constant of the gradient of $f$, both measured with respect to
$I$, where $I$ denotes the identity matrix:
\begin{align}
\mu I\preceq \nabla^2 f(x)\preceq L I. \label{eq:H}
\end{align}
Accordingly, we can define the condition number of the objection function 
\begin{align*}
	\kappa = \frac{L}{\mu}.
\end{align*}

The partial ordering of positive semi-definite matrices is defined in the standard way.
Letting $A$ and $B$ be two $n\times n$ positive semi-definite matrices, we call $A\preceq B$ if $x^\top (B-A)x \ge 0$ for all $x\in\RR^n$.
Given a positive semi-definite matrix $A$, we can define the $A$-norm as $\norm{x}_A = \sqrt{x^\top Ax}$.
We also define the local norm with respect to $x$ as follows:
\begin{equation}
\label{eq:xnorm}
\norm{u}_x =  \norm{x}_{\nabla^2 f(x)}.
\end{equation}
We refer to the inner product of two matrices as follows:
\begin{align}
\label{eq:dp}
\dotprod{A, B} = \tr(A^\top B),
\end{align}
where $\tr(\cdot)$ denotes the trace of a matrix.
$A$ and $B$ can be consistent matrices or vectors in Eqn.~\eqref{eq:dp}.
For a non-singular matrix $A$, we denote its determinant as $\det(A)$.

\section{SR1 Update}
\label{sec:sr1}
Let $A$ and $G$ be two positive definite matrices. Suppose that $A$ is the target matrix and $G$ is the current
approximation of the matrix $A$. The $\sr$ quasi-Newton updates of $G$
with respect to $A$ along a direction $u\in\RR^n\setminus\{0\}$ is the following class of updating
formulas:
\begin{equation}
\label{eq:Gp}
\sr(A,G,u)
=
\left\{
\begin{aligned}
	&G \qquad\qquad\qquad\qquad\qquad\quad\; \mbox{if } (G-A)u = 0,
	\\
	&G- \frac{(G-A)uu^\top (G-A)}{u^\top (G-A)u} \quad \mbox{otherwise}.
\end{aligned}
\right.
\end{equation}

Next, we  present several important properties of the $\sr$ update.
The first property states that each update of $\sr$ preserves the bounds
on the relative eigenvalues with respect to the target matrix.
\begin{lemma}
\label{lem:AGA}
Let  $A$ and $G$ be two positive definite matrices such
that
\begin{equation*}
	A\preceq G \preceq \eta A
\end{equation*}
for some $\eta \ge 1$.
Then for any $u\in\RR^n$, it holds that
\begin{align}
 A\preceq \sr(A,G,u) \preceq G \preceq \eta A. \label{eq:AGA}
\end{align}
\end{lemma}
\begin{proof}
We can assume that $(G-A)u\neq 0$ since otherwise the claim is trivial.
Let us denote $G_+ = \sr(A,G,u)$.
When $A\preceq G\preceq \eta A$, we can obtain that
\[
	G_+ - A 
	\overset{\eqref{eq:Gp}}{=}
	G-A - \frac{(G-A)uu^\top (G-A)}{u^\top(G-A)u}
	=
	(G-A)^{1/2}\left(I - \frac{\tu\tu^\top}{\tu^\top\tu}\right) (G-A)^{1/2}
	\preceq G-A,
\]
where the second equality uses $\tu = (G-A)^{1/2}u$ and last inequality is because of $I - \frac{\tu\tu^\top}{\tu^\top \tu}$ is a projection matrix.
Therefore, we can obtain that
\[
	A\preceq G_+ \preceq G \preceq \eta A.
\]

\end{proof}

We first introduce a potential function which measures the approximation precision of $G$ to $A$.
The potential function is the simple trace potential function, which will be used only when one can guarantee  $A\preceq G$:
\begin{equation}
\label{eq:sig}
\sigma(A, G) \triangleq \tr(G-A) \ge 0.
\end{equation}
The trace potential function has been used to analyze the convergence properties of  greedy SR1 \citep{lin2021faster}.
Based on the trace potential function, the following lemma describes how the $\sr$ update improves the approximation of $A$.
\begin{lemma}
\label{lem:sig}
Let $A \preceq G$ and $u^\top(G-A)u >0 $. Then it holds that
\begin{align*}
	\sigma\left(A,\sr(A,G, u)\right) \le \left(1 - \frac{\lambda_{\min}(G-A)}{\sum_{i=1}^r\lambda_i(G-A)}\right)\cdot \sigma(A,G), 
\end{align*}
where $r$ is the rank of $G -A$ and $\lambda_{\min}(G-A)$ is the smallest non-zero eigenvalue of $G-A$. 
\end{lemma}
\begin{proof}
Let us denote $G_+ \triangleq \sr(A,G,u)$. Then we have
\begin{align*}
	\tr(G_+ - A) 
	=&
	\tr\left(G - A- \frac{(G-A)uu^\top (G-A)}{u^\top (G-A)u}\right)
	\\
	=&
	\tr(G - A) - 
	\frac{u^\top (G-A)^2 u}{u^\top (G-A)u}
	\\
	\le&
	\tr(G-A) - \frac{u^\top (G-A)u}{u^\top u}
	\\
	\le&
	\tr(G-A) - \lambda_{\min}(G-A)
	\\
	=&
	\tr(G-A) - \tr(G-A)/\left(\frac{\sum_{i=1}^r\lambda_i(G-A)}{\lambda_{\min}(G-A)}\right)
	\\
	=&
	\left(1 - \frac{\lambda_{\min}(G-A)}{\sum_{i=1}^r\lambda_i(G-A)}\right)\cdot \tr(G-A),
\end{align*}
where the first inequality is because of Cauchy's inequality that
\begin{align*}
\sqrt{(u^\top (G - A)^2 u)(u^\top u)} \ge  u^\top (G - A)u \ge 0,
\end{align*}
and the second inequality is because of $u^\top(G-A) u > 0$ and the fact
\begin{align*}
\frac{u^\top (G-A)u}{u^\top u} \ge \lambda_{\min}(G-A).
\end{align*}
\end{proof}

Furthermore, the update of SR1 will reduce the rank of $G-A$ but keeps the condition number  $\kappa(G-A) \triangleq \frac{\lambda_{\max}(G-A)}{\lambda_{\min}(G-A)}$ non-increasing.
\begin{lemma}
Let $A \preceq G$ and $u^\top(G-A)u >0 $. Then it holds that
\begin{align*}
\mathrm{rank}(\sr(A,G,u) - A) = \mathrm{rank}(G-A) - 1,
\end{align*}
and
\begin{align*}
\kappa(\sr(A,G,u) - A) \le \kappa(G-A),
\end{align*}
where  $\kappa(G-A) \triangleq \frac{\lambda_{\max}(G-A)}{\lambda_{\min}(G-A)}$ and $\lambda_{\min}(G-A)$ is the smallest non-zero eigenvalue of $G-A$.
\end{lemma}
\begin{proof}
Let us denote $G_+ \triangleq \sr(A,G,u)$.
First, by the SR1 update, we have
\begin{align*}
G_+ - A = G-A - \frac{(G-A)uu^\top (G-A)}{u^\top (G-A)u}.
\end{align*}
Multiplying $u$ to both sides of above equation, we can obtain that
\begin{align*}
(G_+ - A)u = (G-A) u - (G-A)u = 0.
\end{align*}
Since $u^\top(G-A)u > 0$, we can conclude that $\mathrm{rank}(G_+ - A) = \mathrm{rank}(G-A) - 1$.

Also by the SR1 update, we have
\begin{align*}
G - A = G_+ -A + \frac{(G-A)uu^\top (G-A)}{u^\top (G-A)u}.
\end{align*}
Because of $u^\top (G-A) u >0$, $G$ equals $G_+$ plus a rank one positive semi-definite matrix and $\mathrm{rank}(G-A) = \mathrm{rank}(G_+-A) +1$ , then by the interlacing property \citep{horn2012matrix}, we can obtain that 
\begin{align*}
\lambda_{\max}(G-A) \ge \lambda_{\max} (G_+ - A) \ge\dots\ge\lambda_{\min}(G_+ - A) \ge\lambda_{\min}(G-A).
\end{align*}
Therefore, we can obtain that
\begin{align*}
\kappa(G_+ - A) \le \kappa(G-A).
\end{align*}
\end{proof}

We introduce another potential function $V(A, G)$, which plays important roles in our analysis.
\begin{align}
	V(A,G) \triangleq \ln\det\left(GA^{-1}\right),  A,G \succeq 0. \label{eq:tau}
\end{align}
The function $V(A,G)$ has been used to prove the explicit superlinear convergence rate of a class of restrict Broyden quasi-Newton in \citep{rodomanov2021rates}. 
In this paper, we will also use the following measure function to describe the closeness of $ G $ to $ A $ along direction $ u \in \mathbb{R}^n \backslash \{0\} $:
\begin{align}
	\nu(A, G, u) \triangleq \left(\dfrac{u^\top(G-A)G^{-1}(G-A)u}{u^\top(A-AG^{-1}A)u}\right)^{1/2}, G\succeq A. \label{eq:theta}
\end{align}
We can observe that $u^\top (A - AG^{-1}A)u$ in Eqn.~\eqref{eq:theta} is also a factor in the inverse update of SR1 (refer to Eqn.~\eqref{eq:invsr1}).
Thus, the measure function Eqn.~\eqref{eq:theta} is designed only for the SR1 algorithm.
 
Based on the potential function $V(A,G)$ and $\nu(A,G,u)$ defined in Eqn.~\eqref{eq:tau} and \eqref{eq:theta} respectively, the following lemma describes how $\sr$ update improves the approximation of $A$ other than Lemma~\ref{lem:sig}.
\begin{lemma}
	Let $G \succeq A \succ 0$, then for any $u \in \mathbb{R}^n \backslash \{0\}$:
	\begin{align}
		V(A,G) - V(A,\sr(A,G,u)) = \ln\left(1 + \nu^2(A,G,u)\right). \label{eq:dv}
	\end{align}
\end{lemma}
\begin{proof}
	Let us denote $G_+ = \sr(A,G,u)$.
	From SR1 update rule, we can obtain the following inverse update
	\begin{align}
		G_+^{-1} = G^{-1} + \dfrac{(I-G^{-1}A)uu^\top(I-AG^{-1})}{u^\top(A-AG^{-1}A)u}. \label{eq:invsr1}
	\end{align}
Thus, we have
	\begin{align*}
		\det\left(GG_{+}^{-1}\right) &\stackrel{\eqref{eq:invsr1}}{=} \det\left(
		I + \dfrac{G(I-G^{-1}A)uu^\top(I-AG^{-1})}{u^\top(A-AG^{-1}A)u} \right) \\
		&=1+ \dfrac{u^\top(I-AG^{-1})G(I-G^{-1}A)u}{u^\top(A-AG^{-1}A)u} \\
		&=1+ \dfrac{u^\top(G-A)G^{-1}(G-A)u}{u^\top(A-AG^{-1}A)u} \\
		&= 1 + \nu^2(A,G,u).
	\end{align*}
	Thus we have
	\begin{align*}
		V(A,G) - V(A,G_+) = \ln\det\left(GG_{+}^{-1}\right) = \ln\left(1 + \nu^2(A,G,u)\right).
	\end{align*}
\end{proof}

Furthermore, the measure function $\nu(A,G,u)$ also has the following property.
\begin{lemma}
	\label{lem:sig1}
	If $G \succeq A$, it holds that
	\begin{align}
		\nu^2(A,G,u) \ge
		\frac{u^\top(G-A)G_+^{-1}(G-A)u}{u^\top Gu }, \ G_+ = \sr(A,G,u). \label{eq:tt}
	\end{align}
\end{lemma}
\begin{proof}
	If $(G - A)u = 0$, the $\nu(A,G,u) = 0$ and $\frac{u^\top(G-A)G_+^{-1}(G-A)u}{u^\top Gu } = 0$.
	Thus, Eqn.~\eqref{eq:tt} holds trivially.
	If $G \succeq A$ and $Gu\neq Au$, we obtain $u^\top (A - AG^{-1}A)u>0$, then the inequality is well-defined.
	
	Denoting
	$a = u^\top(G-A)G^{-1}(G-A)u$, $b = u^\top(A-AG^{-1}A)u$, then $a+b = u^\top(G-A)u$ and
	\begin{align*}
		u^\top(G-A)G_+^{-1}(G-A)u \overset{\eqref{eq:invsr1}}{=}& u^\top(G-A)G^{-1}(G-A)u + \dfrac{\left(u^\top(G-A)G^{-1}(G-A)u\right)^2}{u^\top(A-AG^{-1}A)u} \\
		=& a+\frac{a^2}{b} = \frac{a(a+b)}{b} = \frac{a}{b} \cdot \left(u^\top(G-A)u\right) \\
		\leq& \frac{a}{b} \cdot \left(u^\top Gu\right) = \nu^2(A,G,u) \cdot \left(u^\top Gu\right),
	\end{align*}
which concludes the proof.
\end{proof}

Lemma~\ref{lem:sig1} plays an important tool in analyzing the explicit convergence rate of the SR1 algorithm for general strongly convex functions.
\section{Unconstrained Quadratic Minimization}
\label{sec:quad}

In this section, we study the vanilla $\sr$  method, as applied to minimizing the quadratic function
\begin{equation}
\label{eq:lsr}
f(x) \triangleq \frac{1}{2} x^\top A x - b^\top x,
\end{equation}
where $A\in\RR^{n\times n}$ is a positive definite matrix and $b\in\RR^n$ is a vector.

\begin{algorithm}[htp]
\caption{$\mathrm{SR1}$ Update for Unconstrained Quadratic Minimization}
\begin{algorithmic}[1]
	\STATE Initialization: Choose $x_0$ and set $G_0 =  L \cdot I$.
	\FOR{$ k = 0,1,\dots , K$}
	\STATE Update $x_{k+1} = x_k - G_k^{-1}\nabla f(x_k)$, 
	\STATE Set $u_k = x_{k+1} - x_k$, 
	\STATE Compute $G_{k+1} = \sr(A,G_k,u_k)$ (By Eqn.~\eqref{eq:Gp}).
	\ENDFOR
\end{algorithmic}
\label{algo:sr1-update}
\end{algorithm}

Consider the  vanilla $\sr$ scheme (Algorithm~\ref{algo:sr1-update}) for minimizing  the problem~\eqref{eq:lsr}.
We assume that the smoothness parameter $L$ is available for  convenience analysis.
In an actual implementation of Algorithm~\ref{algo:sr1-update}, it is typical to store in memory and update in iterations the matrix $H_k \triangleq G_k^{-1}$  instead of $G_k$ (or, alternatively, the
Cholesky decomposition of $G_k$). This allows us to compute $G_{k+1}^{-1} \nabla f(x_{k+1})$ in $O(n^2)$
operations. 
Note that, due to a low-rank structure of the update \eqref{eq:Gp}, $H_k$ can be updated
into $H_{k+1}$ also in $O(n^2)$ operations.

To estimate the convergence rate of Algorithm~\ref{algo:sr1-update}, let us look at the norm of the gradient of $ f(x) $, measured with respect to $ A $:
\begin{equation}\label{eq:lambda}
\lambda_f(x) \triangleq \sqrt{\nabla f(x)^\top A^{-1} \nabla f(x)}, x\in\RR^n.
\end{equation}

It is known that Algorithm~\ref{algo:sr1-update} has at least a linear convergence rate of
the standard gradient method based on $\lambda_{f}(x)$:

\begin{lemma}
In Algorithm~\ref{algo:sr1-update},  it holds that for $k\ge 0$,	
\begin{equation}
	\label{eq:aga}
	A\preceq G_k \preceq \frac{L}{\mu} A~\text{ and }~ \lambda_f(x_k) \le \left(1 - \frac{\mu}{L}\right)^k\lambda_f(x_0).
\end{equation}
\end{lemma}
\begin{proof}
The result $\lambda_f(x_k) \le \left(1 - \frac{\mu}{L}\right)^k\lambda_f(x_0)$ has been proved in Theorem 3.1 of \citet{rodomanov2021rates}. 
The result $A\preceq G_k \preceq \frac{L}{\mu} A$ comes from Lemma~\ref{lem:AGA} and the fact that $A\preceq G_0 = L\cdot I\preceq \frac{L}{\mu} A$.
\end{proof}

\begin{lemma}[Lemma 3.2 of \citet{rodomanov2021greedy} ]\label{lem:lambda}
	Let $k\geq 0$, and let $\eta_k\geq 1$ be such that $A \preceq G_k\preceq\eta_k A$. Then for the quadratic function, we have that 
	$ \lambda_f(x_{k+1})\leq \left(1-\frac{1}{\eta_k}\right)\lambda_f(x_k)\leq (\eta_k-1)\lambda_f(x_k)$. 
\end{lemma}

Furthermore,  the Hessian approximation $G_i$'s and direction $u_i$'s have the following property.
\begin{lemma} \label{lem:uu}
If $u_i^\top (G_i - A) u_i > 0 $ for $i = 0,\dots, k$, where $u_i$ defined in Algorithm~\ref{algo:sr1-update}, then $u_i$'s are linearly independent for $i = 0,\dots, k$.
Furthermore, it holds that 
\begin{align}
Au_i = G_k u_i, \quad i = 0, 1,\dots, k-1. \label{eq:yGu}
\end{align}
\end{lemma} 
\begin{proof}
Denoting $R_k = G_k - A$, by the update of SR1, we can obtain that
\begin{align}
R_k = R_{k-1} - \frac{R_{k-1}u_{k-1}u_{k-1}^\top R_{k-1}}{u_{k-1}^\top R_{k-1} u_{k-1}}. \label{eq:R_def}
\end{align}
and
\begin{align*}
	R_u u_{k-1} = R_{k-1} u_{k-1} - R_{k-1} u_{k-1} = 0.
\end{align*}
Thus, we can obtain that 
\begin{align}
u_{k-1}\in \mathrm{Ker}(R_k) \quad\mbox{and}\quad\mathrm{Ker}(R_{k-1}) \in\mathrm{Ker}(R_k) . \label{eq:ker}
\end{align}

We prove $u_i$'s are linearly independent by contradiction.
Without loss of generality, we assume that $u_k$ can be represented as
\begin{align*}
u_k = \alpha_0 u_0 + \alpha_1 u_1 +\dots+\alpha_{k-1} u_{k-1}.
\end{align*}
Since $\mathrm{Ker}(R_{k-1}) \in\mathrm{Ker}(R_k) $, we have $R_k u_k = R_k \alpha_0 u_0 + \alpha_1 u_1 +\dots+\alpha_{k-1} u_{k-1} = 0+0+\dots+0 = 0$.
This contradicts $u_k^\top (G_k - A)u_k > 0$.

We prove Eqn.~\eqref{eq:yGu} by induction. 
For $i = 0$, we have
\begin{align*}
G_1u_0 
= 
\left(G_0 - \frac{(G_0-A)u_0u_0^\top (G_0-A)}{u_0^\top (G_0-A)u_0}\right)u_0
=G_0 u_0 - (G_0 - A)u_0 
= Au_0.
\end{align*}
Therefore, Eqn.~\eqref{eq:yGu} holds for $i = 0$.

Assuming that Eqn.~\eqref{eq:yGu} holds for some value $k-1 >1$ and show that it holds for $k$.
Using this assumption and Eqn.~\eqref{eq:yGu}, we obtain
\begin{align*}
u_i^\top (Au_{k-1} - G_{k-1} u_{k-1}) 
\overset{\eqref{eq:yGu}}{=}
u_i^\top A u_{k-1} - u_i^\top A u_{k-1} = 0, \quad \mbox{ all } i <k-1.
\end{align*}
Thus, using the update of SR1, we have
\begin{align*}
G_ku_i = G_{k-1}u_i = Au_i, \mbox{ for all } i < k-1.
\end{align*}
The update rule of SR1 guarantees that $G_k u_{k-1} = Au_{k-1}$, 
thus,
Eqn.~\eqref{eq:yGu} holds when $k-1$ is replaced by $k$.
\end{proof}
Now we are going to establish the superlinear convergence of Algorithm~\ref{algo:sr1-update}. 
First, we will work with the trace potential function $\sigma(A, G_k)$  defined by Eqn.~\eqref{eq:sig}. 
Note that this is
possible because $A\preceq G_k$ in view of Eqn.~\eqref{eq:aga}.
\begin{theorem}
\label{thm:lll}
Let $f(x)$ be quadratic and initial Hessian approximation $A\preceq G_0 $. 
If $u_i^\top (G_i - A) u_i > 0 $  for $i = 0,\dots, k$, where $u_i$ defined in Algorithm~\ref{algo:sr1-update}. Then for all $k\ge 0$, the sequence $\{x_k\}$ generated via Algorithm~\ref{algo:sr1-update} satisfies that
\begin{align}
	\lambda_{f}(x_k) 
	\le&
	\prod_{j=1}^{k}\left(1 - \frac{\lambda_{\min}(G_{j-1}-A)}{\sum_{i=1}^{n - j +1} \lambda_i(G_{j-1} -A)}\right)\cdot \frac{\tr(G_0 - A)}{\mu} \cdot \lambda_{f}(x_0) \label{eq:ll_1}
	\\
	\le&
	\prod_{j=1}^{k} \left( 1 - \frac{1}{(n-j+1) \cdot \kappa(G_0 - A)} \right) \frac{\tr(G_0 - A)}{\mu} \cdot \lambda_{f}(x_0). \label{eq:ll_2}
\end{align}

Letting $k > 1$ be the first index such that $u_k^\top (G_k - A) u_k = 0$, then $x_{k+1}$ of Algorithm~\ref{algo:sr1-update} is the minimizer of $f(x)$.
\end{theorem}
\begin{proof}
By the update of SR1, Eqn.~\eqref{eq:R_def} and~\eqref{eq:ker} hold.
Once $u_i^\top (G_i - A) u_i > 0 $  for $i = 0,\dots, k$, Lemma~\ref{lem:uu} shows that $u_i$'s are linear independent for $i = 0,\dots, k$. 
Thus, the dimension of $\mathrm{Ker}(R_k)$ increases at least by $1$ for each iteration. 
By Lemma~\ref{lem:sig}, we can obtain that
\begin{align*}
\sigma(A, G_k) 
\le&
\left(1 - \frac{\lambda_{\min}(G_{k-1}-A)}{\sum_{i=1}^{n - k +1} \lambda_i(G_{k-1} -A)}\right) \cdot \sigma(A,G_{k-1})
\\
\le&
\prod_{j=1}^{k}\left(1 - \frac{\lambda_{\min}(G_{j-1}-A)}{\sum_{i=1}^{n - j +1} \lambda_i(G_{j-1} -A)}\right)\cdot\sigma(A,G_0).
\end{align*}
Furthermore, we have
\begin{align*}
\frac{\lambda_{\min}(G_{j-1}-A)}{\sum_{i=1}^{n - j +1} \lambda_i(G_{j-1} -A)}
\ge
(n-j)^{-1}\frac{ \lambda_{\min}(G_{j-1} - A) }{\lambda_{\max}(G_{j-1} - A) } 
\ge
\frac{1}{(n-j) \cdot \kappa(G_0 - A)}.
\end{align*}

By Lemma~\ref{lem:lambda}, we can obtain that
\begin{align*}
\lambda_{f}(x_k) 
\le&
\prod_{j=1}^{k}\left(1 - \frac{\lambda_{\min}(G_{j-1}-A)}{\sum_{i=1}^{n - j +1} \lambda_i(G_{j-1} -A)}\right)\cdot \frac{\tr(G_0 - A)}{\mu} \cdot \lambda_{f}(x_0)
\\
\le&
\prod_{j=1}^{k} \left( 1 - \frac{1}{(n-j+1) \cdot \kappa(G_0 - A)} \right) \frac{\tr(G_0 - A)}{\mu} \cdot \lambda_{f}(x_0).
\end{align*}

If $u_k^\top (G_k - A) u_k = 0$, then $u_k$ is linearly dependent on $u_i$ for $i = 0, 1,\dots, k-1$ by Lemma~\ref{lem:uu}, 
and $u_k$ can be represented as 
\begin{align*}
u_k = \alpha_0 u_0 + \dots + \alpha_{k-1}u_{k-1}.
\end{align*}
By the definition of $u_k$ in Algorithm~\ref{algo:sr1-update}, we have 
\begin{align*}
G_k^{-1} (\nabla f(x_{k+1}) - \nabla f(x_k))
=&
G_k^{-1} Au_k
\\
=&
\alpha_0 G_k^{-1}Au_0 + \dots+\alpha_{k-1} G_k^{-1} u_{k-1}
\\
\overset{\eqref{eq:yGu}}{=}&
\alpha_0 u_0 + \dots + \alpha_{k-1} u_{k-1}
\\
=&u_k.
\end{align*}
Furthermore, by the update of Algorithm~\ref{algo:sr1-update}, we can obtain that $u_k = - G_k^{-1}\nabla f(x_k)$.
Thus, we can obtain that
\begin{align*}
G_k^{-1} (\nabla f(x_{k_1}) - \nabla f(x_k)) = - G_k^{-1}\nabla f(x_k),
\end{align*}
which, by the nonsingularity of $G_k$, implies that $\nabla f(x_{k+1}) = 0$.
Therefore, $x_{k+1}$ is the solution point.
\end{proof}

\begin{corollary}
	\label{cor:n_step}
Let $f(x)$ be quadratic and initial Hessian approximation $A\preceq G_0 $. Then SR1 find the minimizer of the objective function at most $n$ steps.
\end{corollary}
\begin{proof}
If there exists $k<n$ such that $u_k^\top (G_k - A)u_k = 0$, Theorem~\ref{thm:lll} shows that $x_{k+1}$ is the minimizer.
Thus, SR1 finds the optima no larger than $n$ steps.
If there's no such $k$ that $u_k^\top (G_k - A)u_k = 0$, then when $k = n$, it holds that
\begin{align*}
	\lambda_f (x_n) \overset{\eqref{eq:ll_1}}{\le}&
	\prod_{j=1}^{n-1} \left(1 - \frac{\lambda_{\min}(G_{j-1}-A)}{\sum_{i=1}^{n - j +1} \lambda_i(G_{j-1} -A)}\right) \cdot \left(1 - \frac{\lambda_{\min}(G_{n-1} - A)}{\lambda_{\max}(G_{n-1} - A)}\right)\cdot \frac{\tr(G_0 - A)}{\mu} \cdot \lambda_{f}(x_0)
	\\
	=& 0,
\end{align*}
where the last equality is because $G_{n-1}-A$ is of rank $1$ which implies $\frac{\lambda_{\min}(G_{n-1} - A)}{\lambda_{\max}(G_{n-1} - A)} = 1$.
Thus, SR1 finds the optima at the $n$-th step.
\end{proof}
\begin{remark}
Eqn.~\eqref{eq:ll_1} and \eqref{eq:ll_2} of Theorem~\ref{thm:lll} give the explicit superlinear convergence rate.
Corollary~\ref{cor:n_step} shows that SR1 takes at most $n$ steps to find the minimizer of a quadratic function.
Theorems~\ref{thm:lll} also shows that when $u_k^\top  (G_k -A) u_k = 0$, $x_{k+1}$ in Algorithm~\ref{algo:sr1-update} is the minimizer of quadratic function.  
\end{remark}

\begin{remark}
Theorem~6.1 of \cite{nocedal2006numerical} shows that the vanilla SR1 can solve a quadratic function at most $n$ steps given the condition that $u_k^\top \left(A - AG_k^{-1}A\right) u_k \neq 0$ for all $k$.
In contrast, our results not only show that the SR1 can solve a quadratic function at most $n$ steps but provide an explicit superlinear convergence rate.
Furthermore, the convergence results in Theorem~\ref{thm:lll} does not need extra assumption $u_k^\top \left(A - AG_k^{-1}A\right) u_k \neq 0$ for all $k$. 
\end{remark}
\begin{remark}
Our SR1 algorithm requires the initial Hessian approximation $G_0$ should satisfy $A\preceq G_0$.
This simple restriction can effectively solve the problem that there is no symmetric rank-1 update that satisfies the secant equation.
As a result, our analysis in Theorem~\ref{thm:lll} does \emph{not} require extra assumption that $u_k^\top \left(A - AG_k^{-1}A\right) u_k \neq 0$.
In contrast, once $u_k^\top (G_k - A)u_k = 0$, Theorem~\ref{thm:lll} concludes that  $x_{k+1}$ is minimizer of the objective function without the concern that there is no symmetric rank-1 update that satisfies the secant equation.
\end{remark}

By Lemma~\ref{lem:sig1}, we can obtain another explicit convergence rate of SR1 for the quadratic function.
\begin{theorem}
	\label{thm:l}
	Letting $f(x)$ be quadratic, then for all $k\ge 0$, the sequence $\{x_k\}$ generated via Algorithm~\ref{algo:sr1-update} satisfies that
	\begin{align*}
		\lambda_f(x_k) 
		\le 	
		\left( e^{\frac{n}{k}\ln\frac{L}{\mu}}-1\right)^{k/2}  \sqrt{\frac{L}{\mu}} \cdot\lambda_{f}(x_0)..
	\end{align*}
\end{theorem}
\begin{proof}
	Denote that $V_i = V(A,G_i)$, $\nu_i = \nu(A,G_i,u_i)$, and $g_i \triangleq \norm{\nabla f(x_i)}_{G_i^{-1}}$.
	Then we have
	\begin{align*}
		V_i - V_{i+1} \stackrel{\eqref{eq:dv}}{\geq} \ln(1 + \nu_i^2).
	\end{align*}
	Summing up, we can obtain
	\begin{align*}
		\sum_{i=0}^{k-1} \ln(1 + \nu_i^2) \leq V_0 - V_k \le V_0 =V(A, G_0) \le n \ln \kappa.
	\end{align*}
	The last inequality is because $V_k = \ln\det(A^{-1}G_k)\geq 0$ since $A\preceq G_k$.
	
	Hence, by the convexity of function $t\mapsto \ln(1 + e^t)$, we obtain that
	\begin{align*}
		\frac{n}{k}\ln \kappa 
		& \ge \frac{1}{k} \sum_{i=0}^{k-1}\ln(1 + (\nu_i^2)) = \frac{1}{k} \sum_{i=0}^{k-1}\ln\left(1 + e^{\ln\nu_i^2}\right) \\
		&\ge \ln\left(1+ e^{\frac{1}{k} \sum_{i=0}^{k-1} \ln(\nu_i^2)}\right) = \ln\left(1 + \left[\prod_{i=0}^{k-1} \nu_i^2\right]^{1/k}\right).
	\end{align*}
	Moreover, for all $0\le i\le k-1$, we have
	\begin{align*}
		\nu_i^2 \stackrel{\eqref{eq:tt}}{\geq} \frac{u_i^\top (G_i-A)G_{i+1}^{-1}(G_i-A)u_i}{u_i^\top G_i u_i} =
		\frac{g_{i+1}^2}{g_i^2},
	\end{align*}
	where the last equality is because $G_iu_i = -\nabla f(x_i)$ and $Au_i = \nabla f(x_{i+1}) - \nabla f(x_i)$.
	Hence, 
	\begin{align*}
		\frac{n}{k}\ln \kappa 
		\ge \ln\left(1+ \left[\prod_{i=0}^{k-1} \frac{g_{i+1}^2}{g_i^2} \right]^{2/k}\right) = 
		\ln\left(1+ \left[\frac{g_k}{g_0}\right]^{2/k}\right).
	\end{align*}
	Rearranging, we obtain 
	\begin{align*}
		g_k \le \left( e^{\frac{n}{k}\ln\frac{L}{\mu}}-1\right)^{k/2}g_0.
	\end{align*}

Finally, by the definition of $\lambda_f(x_i)$ and $g_i$, we can obtain that
\begin{align*}
	\lambda_f(x_k) \le \sqrt{\frac{L}{\mu}}\cdot g_k \quad\mbox{ and }\quad g_0 \le \lambda_{f}(x_0),
\end{align*}
since $ A^{-1} \preceq \frac{L}{\mu}G_i^{-1} $ by \eqref{eq:aga}.
Therefore, we can obtain
\begin{align*}
	\lambda_{f}(x_k) 
	\le
	\left( e^{\frac{n}{k}\ln\frac{L}{\mu}}-1\right)^{k/2}  \sqrt{\frac{L}{\mu}} \cdot\lambda_{f}(x_0).
\end{align*}
\end{proof}

We can observe that the SR1 algorithm will converge superlinearly when $k > \frac{n\ln\kappa}{\ln 2}$.
However this result is not meaningful since Theorem~\ref{thm:lll} has shown that SR1 will converge to the optima at most $n$ steps.
On the other hand, Theorem~\ref{thm:l} provides the intuition behind the analysis of the local convergence rate of the SR1 algorithm for the general strongly convex function because a general strongly convex function can be well approximated by a quadratic function in the area near optima.

\section{Minimization of General Functions}
\label{sec:gen}

In this section, we consider a general unconstrained minimization problem defined in Eqn.~\eqref{eq:prob} but without the assumption that $f(x)$ is quadratic.
Besides of the assumption~\eqref{eq:H}, we also assume that $f(x)$ is \emph{strongly self-concordant} with some constant $M\geq 0$, that is,
\begin{align}
\nabla^2 f(y) - \nabla^2 f(x) \preceq M\norm{y-x}_z\nabla^2f(w) \label{eq:M}
\end{align}
for all $x, y, z, w\in\RR^n$.
The class of strongly self-concordant functions is recently introduced by \citet{rodomanov2021greedy} and has been used in the analysis of \citet{lin2021faster,rodomanov2021rates}.
Because the strongly self-concordant functions contain at least all strongly convex functions with Lipschitz continuous Hessian, we conduct our analysis on strongly self-concordant functions.

For the $M$-strongly self-concordant function, it  has the following property.
\begin{lemma}[Lemma 4.1 of \cite{rodomanov2021greedy}]

Let $x,y\in\RR^n$,  and $r = \norm{y-x}_x$. Then,
\begin{align*}
\frac{\nabla^2 f(x)}{1+Mr} \preceq \nabla^2 f(y) \preceq (1+Mr)\nabla^2 f(x).
\end{align*}
And for $J = \int_{0}^1\nabla^2 f(x+t(y-x))\;dt$, we have
\begin{align}
	\left(1+\frac{Mr}{2}\right)^{-1}\nabla^2 f(x)\preceq & J\preceq \left(1+\frac{Mr}{2}\right)\nabla^2f(x), \label{eq:AJA}
	\\
	\left(1+\frac{Mr}{2}\right)^{-1}\nabla^2 f(y)\preceq &J\preceq \left(1+\frac{Mr}{2}\right)\nabla^2 f(y). \label{eq:AJA1}
\end{align}
\end{lemma}

\subsection{SR1 with Correction Strategy}
\label{subsec:sr1_cs}
Now, we need to analyze the properties of the Hessian approximation after a quasi-Newton update.
Letting $G_k$ be the current approximation of $\nabla^2 f(x_k)$, satisfying the condition
\begin{align}
J_k \preceq G_k. \label{eq:prec}
\end{align}
However, even Eqn.~\eqref{eq:prec} holds, the naive update of SR1 using $J_k$ and $u_k$ defined in Algorithm~\ref{algo:sr1-update-1}
\begin{align}
G_{k+1} = \sr(J_k, G_k, u_k), \label{eq:sr11}
\end{align}
may still violate that
\begin{align*}
J_{k+1} \preceq G_{k+1}.
\end{align*}
In fact, $G_{k+1}$ may be not positive definite.
Thus, there maybe exists such $u_{k+1}$ that $u_{k+1}^\top(G_{k+1} - J_{k+1})u_{k+1} = 0$ but $G_{k+1}u_{k+1}\neq J_{k+1}u_{k+1}$ which implies the ill-condition case that there is no symmetric rank-one updating formula satisfying the secant equation. 
This problem is an important reason why the SR1 algorithm may suffer from numerical instability.
More detailed discussion can refer to Chapter 6.2 of \citet{nocedal2006numerical}. 

To conquer the numerical instability, the skipping strategy is very important in the practical implementation of SR1  which skips  the SR1 update when $u_k^\top(G_k - J_k)u_k$ is sufficiently small.
In this paper, to avoid the problem that the condition $G_{k+1} \succeq J_{k+1}$ may be violated  and make the SR1 algorithm more numerically stable, we apply the \emph{correction strategy} introduced in the work of \citet{rodomanov2021greedy}. 
That is, we set $\tG_k = (1+ \delta_k)G_k$ with $\delta_k \ge 0$ and replace $G_k$ with $\tG_k$ in Eqn.~\eqref{eq:sr11} to update $G_{k+1}$.
Specifically, we set 
\begin{align*}
	\delta_k = \left( 1+ \frac{Mr_{k-1}}{2} \right) \left(1 + \frac{Mr_k}{2}\right), \mbox{ with } r_k = \norm{u_k}_{x_k}.
\end{align*}

We describe  our $\sr$ update with correction strategy (referred as \msr) to solve general convex and smooth functions in Algorithm~\ref{algo:sr1-update-1}.
Please note that $J_k$ in Step~\ref{step:J} of Algorithm~\ref{algo:sr1-update-1} is not explicitly  computed but only for the convergence analysis because 
\[
J_ku_k = \nabla f(x_{k+1}) - \nabla f(x_k),
\]
which is just the difference of the successive gradients.
Thus, Algorithm~\ref{algo:sr1-update-1} equals to the vanilla $\sr$ algorithmic procedure except the correction strategy.

\begin{algorithm}[htp]
\caption{{$\mathrm{SR1}$ Update with Correction Strategy for General Unconstrained Minimization} }
\begin{algorithmic}[1]
	\STATE Initialization: Choose $x_0$, and set $G_0 =  L\cdot I$, $r_{-1} = 0$.
	\FOR{$ k = 0,1,\dots, K $}
	\STATE Update $x_{k+1} = x_k - G_k^{-1}\nabla f(x_k)$,
	\STATE Set $u_k = x_{k+1} - x_k$, 
	\STATE Compute $r_k = \norm{u_k}_{x_k}$, and $\tG_k = \left( 1+ \frac{Mr_{k-1}}{2} \right) \left(1 + \frac{Mr_k}{2}\right) G_k$.
	\STATE Denote $J_k = \int_{0}^{1} \nabla^2 f(x_k +tu_k)\;dt$ \label{step:J},
	\STATE Compute $G_{k+1} = \sr(J_k, \tG_k,u_k)$ (By Eqn.~\eqref{eq:Gp}).
	\ENDFOR
\end{algorithmic}
\label{algo:sr1-update-1}
\end{algorithm}

\subsection{Explicit Superlinear Convergence Rate}

To measure the convergence of Algorithm~\ref{algo:sr1-update-1}, we define the local norm of the gradient which is widely used in the convergence analysis of second order methods \citep{rodomanov2021greedy,lin2021faster,boyd2004convex}:
\begin{align}
\label{eq:lam_1}
\lambda_f(x) \triangleq \dotprod{\nabla f(x), [\nabla^2 f(x)]^{-1}\nabla f(x)}^{1/2}.
\end{align}
Note that $\lambda_f(x)$ in above equation is almost the same to the one in Eqn.~\eqref{eq:lambda} except $\nabla^2 f(x)$ equals to $A$.
Using $\lambda_f(x)$, we can estimate the progress of a general quasi-Newton step.

\begin{lemma}[Lemma 4.2 of \cite{rodomanov2021rates}]
	\label{lem:dec}
	Let $f(x)$ be $M$ strongly self-concordant. In Algorithm~\ref{algo:sr1-update-1}, for all $k\ge 0$ and $r_k\triangleq \norm{u_k}_{x_k}$, we have
	\begin{align}
		\lambda_f(x_{k+1}) 
		\le
		\left(1+\frac{Mr_k}{2}\right) \theta\left(J_k,G_k,u_k\right) \lambda_f(x_k). \label{eq:dec}
	\end{align}
\end{lemma}
Furthermore, $r_k$ and $\lambda_{f}(x_k)$ has the following property.
\begin{lemma}
Let $f(x)$ be $M$ strongly self-concordant. 
For $x_k$ and $u_k$ in Algorithm~\ref{algo:sr1-update-1}, denoting $r_k \triangleq \norm{u_k}_{x_k}$ and $r_{-1} = 0$, it holds that
\begin{align}
	r_k \le \left(1+\frac{Mr_{k-1}}{2}\right)\lambda_{f}(x_k). \label{eq:r_l}
\end{align}
\end{lemma}
\begin{proof}
	First, because $\nabla^2 f(x_0) \preceq G_0$, then by the definition of $r_0$, we have
	\begin{align*}
		r_0 
		=& \norm{G_0^{-1}\nabla f(x_0)}_{x_0}
		= \dotprod{\nabla^2 f(x_0), G_{0}^{-1} \nabla^2 f(x_0) G_0^{-1} \nabla f(x_0)}^{1/2} 
		\\
		\le& \dotprod{\nabla^2 f(x_0),  \left(\nabla^2 f(x_0)\right)^{-1}\nabla f(x_0)}^{1/2} 
		= \lambda_{f}(x_0).
	\end{align*}
Similarly, for $k \ge 1$, we have
	\begin{equation*}
	\begin{aligned}
		r_{k} &= \left\|x_{k+1}-x_k\right\|_{x_{k}} =\left\langle\nabla f\left(x_{k}\right), G_{k}^{-1} \nabla^{2} f\left(x_{k}\right) G_{k}^{-1} \nabla f\left(x_{k}\right)\right\rangle^{1 / 2} \\
		&\leq \left(1+\frac{Mr_{k-1}}{2}\right)\left\langle\nabla f\left(x_{k}\right), \nabla^{2} f\left(x_{k}\right)^{-1} \nabla f\left(x_{k}\right)\right\rangle^{1 / 2} =\left(1+\frac{Mr_{k-1}}{2}\right)\lambda_{f}(x_k),
	\end{aligned}
\end{equation*}
where the inequality is because of $G_k = \sr(J_{k-1}, \tG_{k-1}, u_{k-1})$ and
\begin{align*}
	G_k \overset{\eqref{eq:AGA}}{\succeq} J_{k-1} \overset{\eqref{eq:AJA1}}{\succeq} \frac{1}{1+\frac{Mr_{k-1}}{2}}\nabla^2f(x_k).
\end{align*} 
\end{proof}

Next, we will show if the starting point is sufficiently close to the optima, then the relative eigenvalues of the Hessian approximations $G_k$ with respect to the integral Hessians $J_k$ are always located between $1$ and $\frac{L}{\mu}$, up to some small numerical constants.

\begin{theorem}\label{thm:base}
	Suppose $x_0$ in Algorithm~\ref{algo:sr1-update-1} satisfy 
	\begin{align}
	M\lambda_0\leq \frac{\ln\frac{3}{2}}{4\kappa} \mbox{ with } \kappa \triangleq \frac{L}{\mu}. \label{eq:local}
	\end{align}
Denoting $\lambda_k \triangleq \lambda_f(x_k)$, 
	then for all $k\geq 0$, we have
	\begin{align}
		J_{k-1} \preceq& G_k \preceq \xi_k \kappa J_{k-1}, \label{eq:JGJ}
		\\
		e^{Mr_{k-1}/2}\lambda_k \leq& \left(1-\frac{1}{2\kappa}\right)^k \lambda_0 \label{eq:lam_dec}
		\\
		r_{k} \leq& \left(1-\frac{1}{2\kappa}\right)^k \lambda_0, \label{eq:r_dec}
	\end{align}
	and
	\begin{align}
	\xi_k = e^{M \sum_{i=0}^{k-1} \left(r_{i-1}+r_i\right)} \leq \frac{3}{2}, \label{eq:xi}
	\end{align}
where $ r_i = \|u_i\|_{x_i}$, $J_{-1}=\nabla^2 f(x_0)$, $r_{-1}=0$, $\xi_0=1$.
\end{theorem}
\begin{proof}
	We will prove these results by induction.
	For $k = 0$, Eqn.~\eqref{eq:JGJ} holds by $J_{-1} = \nabla^2 f(x_0)$ and $G_0 = L\cdot I$.
	Eqn.~\eqref{eq:r_dec} holds because of Eqn.~\eqref{eq:r_l} and $r_{-1} = 0$.
	Eqn.~\eqref{eq:lam_dec} and \eqref{eq:xi} hold trivially.
	
	We assume that Eqn.~\eqref{eq:JGJ}-\eqref{eq:xi} hold for $ 0 \le j \le k$.
	We first prove Eqn.~\eqref{eq:JGJ} and~\eqref{eq:xi}. 
	We have
	\begin{align*}
		J_{k-1} 
		\overset{\eqref{eq:AJA1}}{\succeq} 
		\left(1 + \frac{Mr_{k-1}}{2}\right)^{-1} \nabla^2 f(x_k) 
		\overset{\eqref{eq:AJA}}{\succeq} \left(1 + \frac{Mr_{k-1}}{2}\right)^{-1} \left(1 + \frac{Mr_k}{2}\right)^{-1} J_k.
	\end{align*}
Similarly, we have
\begin{align}
	J_{k-1} \preceq  	\left(1 + \frac{Mr_{k-1}}{2}\right) \nabla^2 f(x_k)
	\preceq   \left(1+\frac{Mr_{k-1}}{2}\right)\left(1+\frac{Mr_{k}}{2}\right) J_k. \label{eq:JJ1}
\end{align}
Combining with Eqn.~\eqref{eq:JGJ}, we can obtain that
	\begin{equation}
		\label{eq:JGJ_1}
		\frac{1}{\left(1+\frac{Mr_k}{2}\right)\left(1+\frac{Mr_{k-1}}{2}\right)}J_{k} \preceq G_k  \preceq \xi_k \kappa J_{k}\left(1+\frac{Mr_{k-1}}{2}\right)\left(1+\frac{Mr_{k}}{2}\right). 
	\end{equation}
Thus, it holds that $\tG_k \succeq J_k$.
By Lemma~\ref{lem:AGA}, we can obtain that
\begin{equation}
	J_{k} \preceq G_{k+1} \preceq \widetilde{G}_k \preceq \xi_k \kappa J_{k}\left(1+\frac{Mr_{k-1}}{2}\right)^2\left(1+\frac{Mr_{k}}{2}\right)^2 = \xi_{k+1} \kappa J_k. \label{eq:JGJ_2}
\end{equation}
Hence,
\begin{align*}
	\xi_{k+1} 
	\leq&
	\xi_k \left(1+\frac{Mr_{k-1}}{2}\right)^2\left(1+\frac{Mr_{k}}{2}\right)^2
	\le e^{M\sum_{i=0}^{k} \left(r_{i-1}+r_{i}\right)}\xi_0
	\\
	\le&
	e^{2M\sum_{i=0}^{k} r_i} 
	\overset{\eqref{eq:r_dec}}{\le} e^{4\kappa M\lambda_0}
	\overset{\eqref{eq:local}}{\le} \frac{3}{2}.
\end{align*}
Therefore, Eqn.~\eqref{eq:JGJ} and Eqn.~\eqref{eq:xi} are satisfied for $k+1$.

Moreover, by Eqn.~\eqref{eq:JGJ_1}, we can obtain that
\begin{equation*}
	-\alpha_k G_k \preceq G_k-J_k\preceq \beta_k G_k, 
\end{equation*}
with 
\begin{align}
	\alpha_k = \left(1+\frac{Mr_k}{2}\right)\left(1+\frac{Mr_{k-1}}{2}\right)-1, \quad \beta_k = 1- \frac{1}{\kappa\xi_k\left(1+\frac{Mr_{k{-}1}}{2}\right)\left(1+\frac{Mr_k}{2}\right)}. \label{eq:ab}
\end{align}
Thus,
\begin{align*}
\theta\left(J_k,G_k,u_k\right) 
\overset{\eqref{eq:theta}}{=}
\left(
\frac{u_k^\top (G_k - J_k)J_k^{-1}(G_k - J_k) u_k}{u_k^\top G_k J_k^{-1} G_k u_k} 
\right)^{1/2}
\le
\max\left\{\alpha_k, \beta_k\right\}. \label{eq:theta_up}
\end{align*}
By Lemma~\ref{lem:dec}, we have
\begin{align*}
\lambda_{k+1}
\le 
\max\left\{ \alpha_k, \beta_k \right\} \left(1 + \frac{Mr_k}{2}\right) \lambda_k
\le
e^{\frac{Mr_k}{2}} \max\left\{ \alpha_k, \beta_k \right\} \cdot \lambda_k.
\end{align*}
Furthermore, we have
\begin{equation*}
	\begin{aligned}
		e^{M r_{k}} \alpha_k  \overset{\eqref{eq:ab}}{=}& e^{M r_{k}} \left[ \frac{Mr_k}{2}+\frac{Mr_{k-1}}{2}+\frac{Mr_k}{2}\cdot\frac{Mr_{k-1}}{2}\right] \\
		\overset{\eqref{eq:r_dec}}{\le}& e^{M r_{k}} \left(2M\lambda_0\right) \leq 2e^{M\lambda_0} M\lambda_0 
		\overset{\eqref{eq:local}}{\le} \left(\frac{3}{2}\right)^{\frac{1}{4\kappa}} \cdot \frac{\ln\frac{3}{2}}{2\kappa} \leq 1-\frac{1}{2\kappa}.\\
		e^{M r_{k}} \beta_k \overset{\eqref{eq:ab}}{=}&  e^{M r_{k}} \left[1- \frac{1}{\kappa\xi_k\left(1+\frac{Mr_{k{-}1}}{2}\right)\left(1+\frac{Mr_k}{2}\right)}\right] \\
		\leq& e^{M\lambda_0} \left[1-\frac{2}{3\kappa e^{M\lambda_0}}\right] = e^{M\lambda_0}-\frac{2}{3\kappa} \\ 
		\overset{\eqref{eq:local}}{\le}& \left( \frac{3}{2} \right)^{1/4\kappa} -\frac{2}{3\kappa} \leq 1+\frac{1}{8\kappa}-\frac{2}{3\kappa} \leq 1-\frac{1}{2\kappa}.
	\end{aligned}
\end{equation*}
Therefore,
\begin{align}
e^{Mr_k/2}\lambda_{k+1} \leq e^{M r_{k}} \max\{\alpha_k, \beta_k\}  \lambda_{k} \leq \left(1-\frac{1}{2\kappa}\right)\lambda_k \leq \left(1-\frac{1}{2\kappa}\right)e^{Mr_{k-1}/2}\lambda_k \leq \left(1-\frac{1}{2\kappa}\right)^{k+1}\lambda_0. \label{eq:lam_dec_1}
\end{align}
That is Eqn.~\eqref{eq:lam_dec} holds for $k+1$.
	
Finally, we will prove Eqn.~\eqref{eq:r_dec} holds for $k+1$ and we have
\begin{align*}
r_{k+1} 
\overset{\eqref{eq:r_l}}{\le}
\left(1 + \frac{M r_k}{2}\right) \lambda_f(x_{k+1})
\le
e^{M r_k/2}\lambda_f(x_{k+1}) 
\overset{\eqref{eq:lam_dec_1}}{\le}
\left(1-\frac{1}{2\kappa}\right)^{k+1}\lambda_0.
\end{align*}
Thus, Eqn.~\eqref{eq:r_dec} also holds for $k+1$.
This concludes the proof.
\end{proof}

Theorem~\ref{thm:base} directly implies that $J_k \preceq \tG_k$ which is the key in our correction strategy.
\begin{corollary}
Suppose $x_0$ in Algorithm~\ref{algo:sr1-update-1} satisfy Eqn.~\eqref{eq:local}, then it holds that
\begin{align}
J_k \preceq \tG_k, \quad G_k \preceq 3 \kappa \cdot \nabla^2 f(x_k).  \label{eq:JtG}
\end{align}
\end{corollary}
\begin{proof}
First, we have
\begin{align*}
	J_{k-1} 
	\overset{\eqref{eq:AJA1}}{\succeq} 
	\left(1 + \frac{Mr_{k-1}}{2}\right)^{-1} \nabla^2 f(x_k) 
	\overset{\eqref{eq:AJA}}{\succeq} \left(1 + \frac{Mr_{k-1}}{2}\right)^{-1} \left(1 + \frac{Mr_k}{2}\right)^{-1} J_k.
\end{align*}
Combining with Eqn.~\eqref{eq:JGJ}, we can obtain that
\begin{equation*}
	\frac{1}{\left(1+\frac{Mr_k}{2}\right)\left(1+\frac{Mr_{k-1}}{2}\right)}J_{k} \preceq G_k. 
\end{equation*}
Using the definition $\tG_k \triangleq \left(1+\frac{Mr_k}{2}\right)\left(1+\frac{Mr_{k-1}}{2}\right) G_k$, we can obtain the first result.

Second, by Eqn.~\eqref{eq:JGJ} and~\eqref{eq:JJ1}, we can obtain
\begin{align*}
G_k \stackrel{\eqref{eq:JGJ}}{\preceq} \xi_k\kappa J_{k-1} 
\stackrel{\eqref{eq:JJ1}}{\preceq } \xi_k \left(1 + \frac{Mr_{k-1}}{2}\right) \kappa \nabla^2 f(x_k)
\stackrel{\eqref{eq:xi}\eqref{eq:r_dec}\eqref{eq:local}}{\preceq } 3\kappa \cdot \nabla^2 f(x_k),
\end{align*}
which concludes the proof.
\end{proof}
Based on above results, we are ready to give the superlinear convergence rate of SR1 with correction strategy (Algorithm~\ref{algo:sr1-update-1}).

\begin{theorem}
	\label{thm:ll}
	Suppose that the initial point $x_0$ is chosen sufficiently
	close to the solution, 
	\[ M\lambda_0\leq \frac{\ln\frac{3}{2}}{4\kappa}. \] 
	Then, for all $k \geq 1$, Algorithm~\ref{algo:sr1-update-1} satisfies
	\begin{equation}
		\lambda_f(x_k) \le \left(e^{\frac{2n \ln (e \kappa)}{k}}-1\right)^{k/2}\sqrt{3\kappa} \cdot \lambda_f(x_0). \label{eq:sr}
	\end{equation}
\end{theorem}

\begin{proof}
Denote  
\begin{align}
a_k \triangleq \left(1+\frac{Mr_{k-1}}{2}\right)\left(1+\frac{Mr_k}{2}\right) \mbox{ and } g_k \triangleq \norm{\nabla f(x_k)}_{G_k^{-1}}. \label{eq:g}
\end{align}
Then it holds that
\begin{align}
a_k-1 = \frac{Mr_{k-1}}{2} + \frac{Mr_{k}}{2} + \frac{Mr_{k-1}}{2}\cdot \frac{Mr_{k}}{2}
\overset{\eqref{eq:r_dec}}{\le} \left(1-\frac{1}{2\kappa}\right)^{k-1} 2M\lambda_0 \leq \left(1-\frac{1}{2\kappa}\right)^{k-1} \frac{\ln\frac{3}{2}}{2\kappa}. \label{eq:ak}
\end{align}
Thus, $a_k$ satisfies 
\begin{align}
	1 \le a_k \le \frac{3}{2}. \label{eq:ak12}
\end{align}

First, by Eqn.~\eqref{eq:JtG}, we can obtain that $\tG_k \succeq J_k$.
By Lemma~\ref{lem:AGA}, we can obtain that $G_{k+1} \succeq J_k$.
Thus, we can obtain that
\begin{align*}
	V(J_k, \tG_k) - V(J_k, G_{k+1}) 
	\overset{\eqref{eq:dv}}{=}
	\ln \left(1+ \nu(J_k, \tG_k, u_k)^2\right). 
\end{align*}
Moreover, note that
\begin{align*}
	\nu(J_k, \tG_k, u_k)^2 & \overset{\eqref{eq:tt}}{\ge}
	\frac{u_k^\top (\tG_k - J_k)G_{k+1}^{-1} (\tG_k - J_k)}{u_k^\top \tG_k u_k} \\
	& =
	\frac{u_k^\top (G_k - J_k) G_{k+1}^{-1} (G_k - J_k)u_k}{u_k^\top \tG_k u_k}
	+
	2(a_k-1)\frac{u_k^\top G_k G_{k+1}^{-1} (G_k - J_k) u_k }{u_k^\top \tG_k u_k} \\
	&\quad +
	\frac{(a_k -1)^2u_k^\top G_k^\top G_{k+1}^{-1} G_k u_k}{u_k^\top \tG_k u_k}
	\\
	& \ge \frac{u_k^\top (G_k - J_k) G_{k+1}^{-1} (G_k - J_k)u_k}{u_k^\top \tG_k u_k} - 2(a_k - 1) \frac{u_k^\top G_kG_{k+1}^{-1}J_k u_k}{u_k^\top \tG_k u_k} \\
	&\quad +(a_k^2-1)\frac{u_k^\top G_k^\top G_{k+1}^{-1} G_k u_k}{u_k^\top \tG_k u_k} \\
	&\geq \frac{u_k^\top (G_k - J_k) G_{k+1}^{-1} (G_k - J_k)u_k}{u_k^\top \tG_k u_k} - 2(a_k - 1) \frac{u_k^\top G_kG_{k+1}^{-1}J_k u_k}{u_k^\top \tG_k u_k}.
\end{align*}
Furthermore, from SR1 update in Algorithm \ref{algo:sr1-update}, we have $G_{k+1}u_k=J_ku_k$, thus
\begin{align*}
	u_k^\top \tG_k G_{k+1}^{-1} J_k u_k = u_k^\top \tG_k u_k.
\end{align*}
Therefore, we can obtain that 
\begin{align*}
	\nu(J_k, \tG_k, u_k)^2 \ge \frac{u_k^\top (G_k - J_k) G_{k+1}^{-1} (G_k - J_k)u_k}{a_k u_k^\top G_k u_k} - \frac{2(a_k - 1)}{a_k} = \frac{g_{k+1}^2}{a_k g_k^2}-\frac{2(a_k - 1)}{a_k}.
\end{align*}
where the last equality is because $G_k u_k = -\nabla f(x_k)$, $J_ku_k = \nabla f(x_{k+1}) - \nabla f(x_k)$ and definition of $g_k$ (by Eqn.~\eqref{eq:g}).

Consequently, we can obtain that
\begin{align}
	V(J_k, \tG_k) - V(J_k, G_{k+1}) \geq \ln\left(1+\frac{1}{a_k}\cdot \frac{g_{k+1}^2}{g_k^2}-\frac{2(a_k - 1)}{a_k}\right).\label{eq:p1}
\end{align}

Furthermore, we have
\begin{align}
	V(J_k, G_{k+1}) - V(J_k, \tG_{k+1}) = \ln\det(J_k^{-1}G_{k+1}) - \ln\det(a_{k+1}J_{k}^{-1}G_{k+1})
	= -n\ln (a_{k+1}),  \label{eq:p2}
\end{align}
and
\begin{equation}\label{eq:p3}
	V(J_k,\tG_{k+1}) - V(J_{k+1}, \tG_{k+1}) = \ln\det(J_k^{-1}\tG_{k+1}) - \ln\det(J_{k+1}^{-1}\tG_{k+1}) 
	\stackrel{\eqref{eq:JJ1}}{\geq} -n\ln(a_{k+1}).
\end{equation}
Therefore, we obtain
\begin{equation*}
	\begin{aligned}
		&V(J_k, \tG_k) - V(J_{k+1}, \tG_{k+1}) 
		\\
		=&V(J_k, \tG_k) - V(J_k, G_{k+1})
		+ V(J_k, G_{k+1}) - V(J_k, \tG_{k+1}) 
		\\
		& + V(J_k,\tG_{k+1}) - V(J_{k+1}, \tG_{k+1}) \\
		\overset{\eqref{eq:p1}\eqref{eq:p2}\eqref{eq:p3}}{\geq}&
		\ln\left(1+\frac{1}{a_k}\cdot \frac{g_{k+1}^2}{g_k^2}-\frac{2(a_k - 1)}{a_k}\right)-2n\ln(a_{k+1}) \\
		=& \ln\left(2-a_k + \frac{g_{k+1}^2}{g_k^2}\right)-2n\ln(a_{k+1}) -\ln a_k \\
		\stackrel{\eqref{eq:ak12}}{\geq}& \ln\left(\frac{1}{2} + \frac{g_{k+1}^2}{g_k^2}\right)-2n\ln(a_{k+1}) -\ln a_k \\
		=& \ln\left(1 + \frac{2g_{k+1}^2}{g_k^2}\right)-2n\ln(a_{k+1}) -\ln (2a_k).
	\end{aligned}
\end{equation*}
Summing up the above equation, we can obtain that
\begin{align*}
	& \sum_{i=0}^{k-1} \ln\left(1 + \frac{2g_{i+1}^2}{g_i^2}\right)\\
	\leq& \sum_{i=0}^{k-1} \left[ V(J_i, \tG_i) - V(J_{i+1}, \tG_{i+1}) + 2n\ln(a_{i+1})+\ln (2a_i)\right] \\
	=& V(J_0, \tG_0) - V(J_k, \tG_k) + 2n\sum_{i=1}^{k}(a_i - 1)+n\ln 2 +\sum_{i=0}^{k-1} (a_i-1)
	\\
	\stackrel{\eqref{eq:ak}}{\le}& V(J_0, \tG_0) + 2n \ln\frac{3}{2}+n\ln 2+\ln \frac{3}{2}\\
	\leq& V(J_0, \tG_0) + 2n.
\end{align*}
Since $\mu \cdot I \preceq J_0 $, and $\tG_0 = \left(1 + \frac{Mr_0}{2}\right)G_0 \le \frac{9L}{8} \cdot I$, we can  obtain that
\begin{align*}
	\sum_{i=0}^{k-1} \ln\left(1 + \frac{2g_{i+1}^2}{g_i^2}\right)
	\leq \frac{9n}{8}\ln\kappa + 2n \leq 2n \ln (e \kappa).
\end{align*}
By the convexity of function $t \mapsto \ln(1+e^t)$, it holds that
\begin{equation*}
	\begin{aligned}
		\frac{2n \ln (e \kappa)}{k} &\geq \frac{1}{k}\sum_{i=0}^{k-1} \ln\left(1 + \frac{2g_{i+1}^2}{g_i^2}\right) = 
		\frac{1}{k} \sum_{i=0}^{k-1} \ln\left( 1 + e^{\ln \left(2g_{i+1}^2/g_i^2\right) } \right)
		\\ &\ge
		\ln\left( 1+ e^{\frac{1}{k} \sum_{i=0}^{k-1}\ln \left(2g_{i+1}^2/g_i^2\right) } \right) =
		\ln\left( 1+ e^{\frac{1}{k}\ln \left(2^kg_{k}^2/g_0^2\right) } \right) \\
		&= \ln\left( 1+ 2\left[\frac{g_{k}^2}{g_0^2}\right]^{1/k} \right)
	\end{aligned}
\end{equation*}
Rearranging the above equation, we can obtain that
\begin{align*}
	g_k \le \left(e^{\frac{2n \ln (e \kappa)}{k}}-1\right)^{k/2} g_0.
\end{align*}

Finally, we have
\begin{align*}
	\lambda_f(x_k) \stackrel{\eqref{eq:JtG}}{\le} \sqrt{3\kappa} \cdot g_k, \mbox{ and } g_0 \le \lambda_f(x_0).
\end{align*}
Therefore, we can obtain that
\begin{align*}
\lambda_f(x_k) \le \left(e^{\frac{2n \ln (e \kappa)}{k}}-1\right)^{k/2}\sqrt{3\kappa} \cdot \lambda_f(x_0).
\end{align*}
\end{proof}

Next, we give the following corollary which provides a more clear convergence rate description.
\begin{corollary}
For all $k > \frac{4n\ln(e\kappa)}{\ln 2}$, Algorithm~\ref{algo:sr1-update-1} satisfies that
\begin{align}
	\lambda_{f}(x_k) \le \left( \frac{4n\ln(e\kappa)}{k} \right)^{k/2} \sqrt{3\kappa} \cdot \lambda_f(x_0). \label{eq:sr1}
\end{align}
\end{corollary}
\begin{proof}
	Since $e^t \le \frac{1}{1 - t} = 1 + \frac{t}{1 - t}$ for all $t < 1$, we have, for all $k > \frac{4n\ln(e\kappa)}{\ln 2}$,
	\begin{align*}
		e^{\frac{2n\ln(e\kappa)}{k}} - 1 \le \frac{2n\ln(e\kappa)/k}{1 - 2n\ln(e\kappa)/k} \le \frac{4n\ln(e\kappa)}{k}.
	\end{align*}
Combining with Eqn.~\eqref{eq:sr}, we obtain the result.
\end{proof}

\begin{remark}
The main proofs in this section are similar to ones of \citet{rodomanov2021new} and we adopt the main idea of \citet{rodomanov2021new}  which proved convergence rates of quasi-Newton in the restricted Broyden family to prove the convergence rate of the SR1 algorithm.
However, our work is \emph{not} a simple extension of \citet{rodomanov2021new}.

First and the most important, because the SR1 algorithm does \emph{not} belong to the restricted Broyden family, the nice properties holding for quasi-Newton in the restricted Broyden family may no longer hold for the SR1 algorithm.
This is the reason why there's even no local superlinear result of SR1  similar with the ones of BFGS and DFP \citep{nocedal2006numerical}. 
A simple extension of \citet{rodomanov2021new} to the SR1 algorithm can not obtain the results in this paper.
Thus, we introduce the measure function Eqn.~\eqref{eq:theta} whose key ingredient is a factor in the inverse update of SR1.
This is the key to our convergence analysis and is the main difference between the proofs in this paper and ones of \citet{rodomanov2021new}.

Second, to conquer the problem that the SR1 algorithm may suffer from the ill-posed case that  there is no symmetric rank-one updating formula satisfying the secant equation, we introduce the correction strategy which is not used in the classical quasi-Newton methods in \citet{rodomanov2021new}.
Accordingly, this causes several important differences in proofs compared with the ones of \citet{rodomanov2021new}.   
\end{remark}

\subsection{Stability of SR1 with Correction Strategy}

In Section~\ref{subsec:sr1_cs}, we give the reason why we need to introduce the correction strategy.
Now, we will mathematically show the advantages of correction strategy.
Even for a convex quadratic function, the vanilla SR1 algorithm maybe have steps on which there is no symmetric rank-1 update that satisfies the secant equation. 
That is, it holds that $(J_k - G_k) u_k \neq 0$ but $u_k^\top  (J_k - G_k) u_k = 0$ (refer to Chapter 6.2 of \citet{nocedal2006numerical}).
This will cause numerical instabilities and even breakdown of the SR1.
However, this unwanted situation will not happen for   our \msr. 
This is because the positive semi-definiteness of $\tG_k - J_k$ in  \msrs guarantees that once $u_k^\top (\tG_k - J_k)u_k = 0$, then it holds that $ G_ku =  J_k u_k$, that is the updating formula is simply $G_{k+1} = \tG_k$.
We summarize above propositions as follows.
\begin{proposition}
	\label{prop:def}
If $u_k^\top (\tG_k - J_k) u_k = 0$, then it holds that $(\tG_k - J_k) u_k = 0$. 
The SR1 update rule conducts that $G_{k+1} = \tG_k$. 
\end{proposition}
\begin{proof}
	Since $\tG_k - J_k$ is positive semi-definite, then we and represent that $\tG_k - J_k = LL^\top$.
	Then $u_k^\top  (\tG_k - J_k) u_k = 0$ implies  $u_k^\top L L^\top u_k = 0$ which leads to $L^\top u_k = 0$.
	By Theorem~7.2.7 of \citet{horn2012matrix}, we can obtain that $(\tG_k - J_k) u_k = 0 $.
	In this case, by Eqn.~\eqref{eq:Gp}, we can obtain that $G_{k+1} = \sr(J_k, \tG_k, u_k) = \tG_k$.
\end{proof}

\section{Numerical Experiments}
\label{sec:exp}
In Section~\ref{sec:gen}, we propose to use the correction strategy to make the vanilla SR1 more numerically stable.
Thus, we will empirically validate this point and study how correction strategy affects the convergence properties of  SR1 in this section.

We conduct experiments on the widely used logistic regression defined as follows:
\begin{equation*}
	f(x) = \frac{1}{m}\sum_{i=1}^{m} \log [1+\exp(-b_i\langle a_i, x\rangle)] + \frac{\gamma}{2}\|x\|^2, \label{eq:rlg}
\end{equation*}
where $a_i \in \RR^{n}$ is the $i$-th input vector,  $b_i\in\{-1,1\}$ is the corresponding label, and $\gamma \ge 0$ is the regularization parameter.
In our experiments, we set $\gamma = \frac{1}{10m}$.
We conduct experiments on four datasets `mushrooms', `a9a', `w8a', and `madelon'.
Because the exact value $M$ in Eqn.~\eqref{eq:M} is commonly unknown, we set different values of $M$ to evaluate how correction strategy affects the convergence rate of our modified SR1 (referred as \msr).
We compare \msrs of different $M$'s with the vanilla SR1 algorithm (referred as \emph{SR1}).
To simulate the local convergence, we use the same initialization after running three standard Newton steps to make $f(x_0) - f(x_*)$ small enough.

\begin{figure*}[t]
	\subfigtopskip = 0pt
	\begin{center}
		\centering
		\subfigure[`a9a']{\includegraphics[width=75mm]{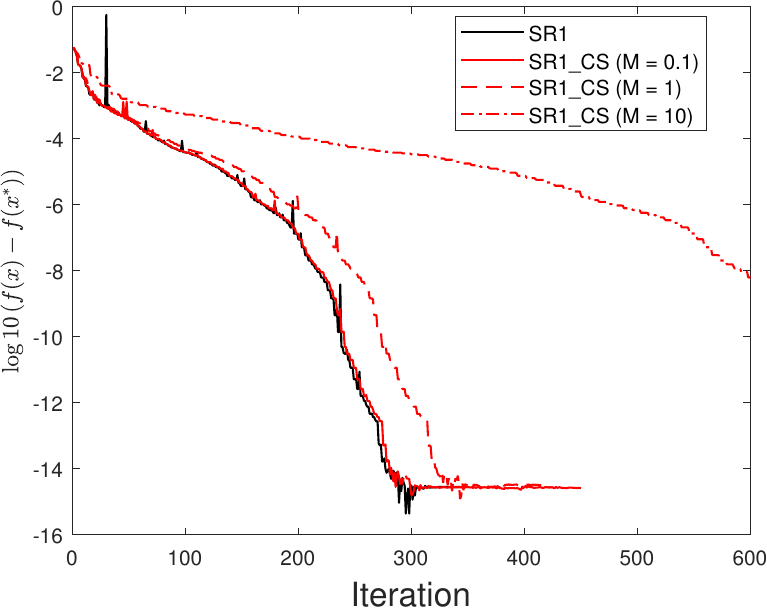}}~
		\subfigure[`madelon']{\includegraphics[width=75mm]{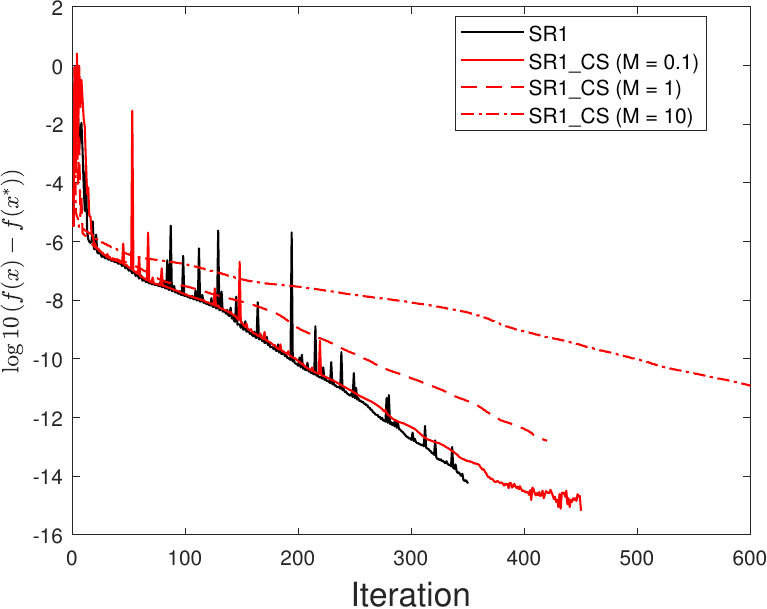}}
		\\
		\subfigure[`mushrooms']{\includegraphics[width=75mm]{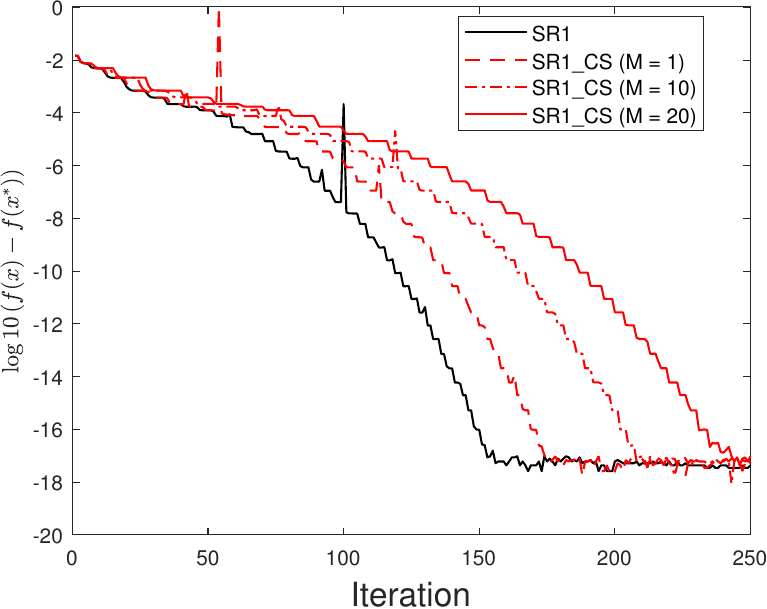}}~
		\subfigure[`w8a']{\includegraphics[width=75mm]{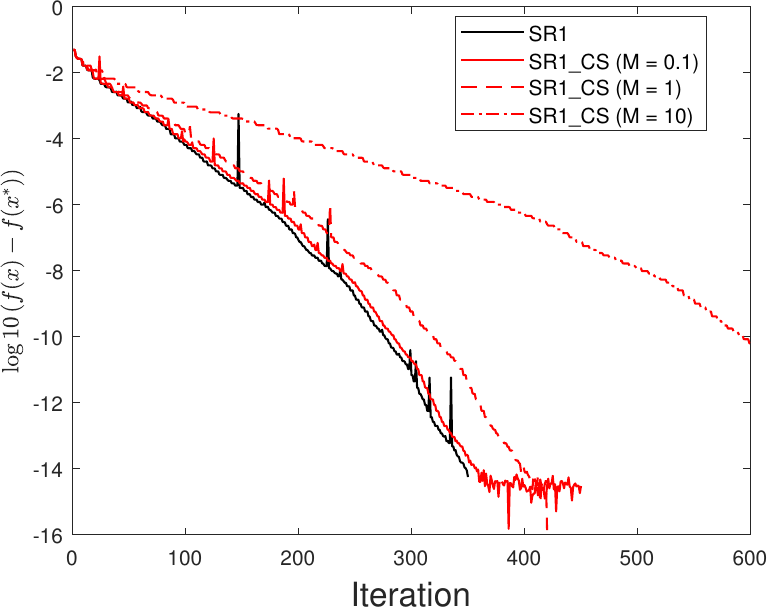}}
	\end{center}
	\vskip -0.2in
	\caption{Comparison between SR1 and \msr}
	\vskip -0.2in
	\label{fig:avazu}
\end{figure*} 

We report the experiment result in Figure~\ref{fig:avazu}. 
We can observe that our modified SR1 algorithm with the correction strategy can effectively improve the numerical stability of SR1.
This phenomenon is more obvious when $M$ is large.
By the correction strategy, \msrs commonly converges stably especially when $x_k$ is close to the optima.
In contrast, vanilla SR1 suffers from instabilities even $x_k$ is close to the optima.
This can be clearly observed from Figure~\ref{fig:avazu}.
 
Figure~\ref{fig:avazu} also shows that the numerical stability brought by the correction strategy is  at the expense of the fast  convergence rate. 
A large $M$ commonly leads to a stable convergence property.
However, a large $M$ will lead to a slow convergence rate.
Whatever, our\msrs can still achieve superlinear convergence rates no matter the value of $M$.
Our experiment results show that for the logistic regression, $M = 1$ is commonly a good choice which can help to achieve numerical stability but keeps a fast convergence rate.

\section{Discussion}
\label{sec:disc}

Let us compare the convergence rates obtained in this paper with the previously known ones of quasi-Newton recently obtained in \citep{rodomanov2021greedy,rodomanov2021new,rodomanov2021rates,lin2021faster}.  
We mainly discuss the general strongly convex function case since for the quadratic function, all SR1 algorithms run at most $n$ steps while other quasi-Newton methods such as DFP and BFGS commonly require much more steps.

First, let us consider the starting moment of superlinear convergence.
Our SR1 algorithm with the correction strategy starts its superlinear convergence after 
\begin{align}
	K_0^{\mathrm{SR1\_CS}} = \frac{2n \ln (e\kappa)}{\ln 2}, \label{eq:K_sr1}
\end{align}
steps by Theorem~\ref{thm:ll}.
We first compare our result with the greedy and randomized quasi-Newton methods proposed and analyzed in \citep{rodomanov2021greedy,lin2021faster}.
These methods also apply the correction strategy. 
For the greedy SR1 and randomized SR1, their starting moment is \citep{lin2021faster}
\begin{align}
K_0^{\mathrm{greedy\_SR1}} =  4\max\{n,\kappa\} \ln(2n\kappa). \label{eq:K_sr0}
\end{align}
For the randomized BFGS (rand\_BFGS), its starting moment is \citep{lin2021faster}
\begin{align}
K_0^{\mathrm{rand\_BFGS}} = \max\left\{n, 2\kappa  \right\}\ln(4n\kappa). \label{eq:K_rbfgs}
\end{align} 

Comparing Eqn.~\eqref{eq:K_sr1} with \eqref{eq:K_sr0} and \eqref{eq:K_rbfgs}, we can observe that our \msrs starts much earlier than the greedy SR1 and randomized BFGS if the condition number $\kappa$ is much larger than the dimension $n$.

Now, we consider the starting moment for classical BFGS and DFP obtained \citet{rodomanov2021rates} but improved in \citet{rodomanov2021new}:
\begin{align}
	K_0^{\mathrm{BFGS}} = 8n \ln (2\kappa),\quad\mbox{ and }\quad K_0^{\mathrm{DFP}} = 18n\kappa \ln(2\kappa). \label{eq:K_bfgs}
\end{align}
Comparing Eqn.~\eqref{eq:K_sr1} with \eqref{eq:K_bfgs}, the starting moment of \msrs is a little earlier than the classical BFGS but much earlier than classical DFP.

Now, we will discuss the convergence rates. 
We will mainly compare the convergence rate of \msrs with the ones of the classical BFGS and greedy SR1 because classical DFG converges much slower than BFGS and the greedy SR1 also outperforms or is comparable to other greedy and randomized quasi-Newton methods \citep{lin2021faster}. 
The classical BFGS has the following convergence rate \citep{rodomanov2021new}:
\begin{align*}
	\lambda_f(x_k) 
	\le
	\left(\frac{5}{2} \left(e^{\frac{13n\ln(2\kappa)}{6k}} - 1\right) \right) ^{k/2} \sqrt{\frac{3\kappa}{2}} \cdot \lambda_f(x_0).
\end{align*}
Comparing the above equation with Eqn.~\eqref{eq:sr}, we can conclude that \msrs has a comparable convergence rate with the classical BFGS.

For the greedy SR1, \citet{lin2021faster} gave the following convergence rate:
\begin{align*}
	\lambda_{f}(x_{k+1}) 
\le
2n\kappa^2\left(1 - \frac{1}{n}\right)^k \lambda_{f}(x_k),
\end{align*}
which implies that
\begin{align*}
\lambda_f(x) \le (2n\kappa^2)^k \cdot e^{-\frac{k(k-1)}{2n}} \lambda_f(x_0). \label{eq:gr}
\end{align*}
We will compare the above rate with the one in Eqn.~\eqref{eq:sr1}.
We have
\begin{equation*}
	\begin{aligned}
		\frac{2n\kappa^2\left(1 - \frac{1}{n}\right)^k}{\left(\frac{4n \ln (e \kappa)}{k}\right)^{k/2} \sqrt{3\kappa}} &= \exp\left\{-\frac{k(k-1)}{2n}+k\ln(2n\kappa^2)-\frac{k}{2}\ln\frac{4n \ln (e \kappa)}{k} -\frac{\ln 3\kappa}{2} \right\} \\
		&= \exp\left\{-\frac{k}{2}\left(\frac{k-1}{n}-\ln k\right)+ \frac{k}{2}\left(\ln(4n^2\kappa^4)-\ln\left(4n \ln (e \kappa)\right)\right) -\frac{\ln 3\kappa}{2} \right\} \\
		&= \exp\left\{-\frac{k}{2}\left(\frac{k-1}{n}-\ln k\right)+ \frac{k}{2}\ln\left(\frac{n\kappa^4}{\ln\kappa+1}\right)-\frac{\ln 3\kappa}{2} \right\} \\
		&\geq \exp\left\{-\frac{k}{2}\left(\frac{k-1}{n}-\ln k\right)+ \frac{k}{2}\left[\ln\left(\frac{n\kappa^4}{\kappa}\right)-\ln 3\kappa \right]\right\} \\
		&\geq \exp\left\{-\frac{k}{2}\left(\frac{k-1}{n}-\ln k\right)+ \frac{k}{2}\ln\left(\frac{n\kappa^2}{3}\right)\right\} \\
		&=\exp\left\{-\frac{k}{2}\left(\frac{k-1}{n}-\ln k-\ln\left(\frac{n\kappa^2}{3}\right)\right)\right\}.
	\end{aligned}
\end{equation*}
Note that $t-\ln(tn+1)$ is increasing when $t\geq 1$. 
Thus when $$k \geq K\triangleq 3n\ln (2n\kappa)+1,$$ that is $(k-1)/n \geq 1$. 
It holds that
\begin{equation*}
	\begin{aligned}
		\frac{k-1}{n}-\ln k &= t-\ln(tn+1)-\ln\left(\frac{n\kappa^2}{3}\right) \\
		&\geq 3\ln2n\kappa-\ln\left[3n\ln(2n\kappa)+1\right]-\ln\left(\frac{n\kappa^2}{3}\right) \\
		&= \ln\frac{24n^2\kappa}{3n\ln(2n\kappa)+1} \geq \ln\frac{24n^2\kappa}{3n\ln(2n\kappa)+n} \\
		&= \ln\frac{8n\kappa}{\ln(2n\kappa)+1/3} \geq \ln\frac{24n^2\kappa}{2n\kappa} = \ln (12n) > 0. 
	\end{aligned}
\end{equation*}
This implies  
\begin{align*}
\frac{2n\kappa^2\left(1 - \frac{1}{n}\right)^k}{\left(\frac{4n \ln (e \kappa)}{k}\right)^{k/2} \sqrt{3\kappa}} < 1.
\end{align*}
Therefore, once the greedy SR1 begins to achieve the superlinear convergence rate, that is $k\geq K^{\mathrm{greedy\_SR1}}_0\geq K$, the greed SR1 converges faster than our \msr.
\section{Conclusion}
\label{sec:conc}

In this paper,  we have studied the famous quasi-Newton method: SR1 update, and presented the  explicit superlinear convergence rate of the SR1 algorithm which only involves the gradients of the objective function for the first time,  to the best of our knowledge.
We have shown that the SR1 algorithm with correction strategy has a convergence rate of the form $\left(\frac{4n\ln(e\kappa)}{k}\right)^{k/2}$ for general smooth strongly convex functions.
We also show that the vanilla SR1 algorithm also achieves the superlinear rate and will find the optima at most $n$ steps for quadratic functions with initial Hessian approximation satisfying $\nabla^2 f(x)\preceq G_0$.
In this paper, the analysis requires that $G_0 \succeq \nabla^2 f(x_0)$ to obtain the explicit superlinear convergence rate.
However, in real applications, even this condition is violated, vanilla SR1 commonly can achieve good performance.
To analyze this case, we leave it as the future work.

\bibliography{ref.bib}
\bibliographystyle{apalike2}

\pagebreak
\appendix

\end{document}